\documentclass[12pt,a4paper]{amsart}
\usepackage{amsmath}
\usepackage{drpack}
\usepackage[english]{babel}
\usepackage{verbatim}
\usepackage[shortlabels]{enumitem}
\usepackage{tikz-cd}
\usepackage{pgfplots}
\usepackage{url}

\newtheoremstyle{mio}%
	{}{} % spazio sopra e sotto%
	{\itshape}{} % corpo del testo, indentation
	{\bfseries}{.}{ } % titolo del teorema: tipo di testo, divisore, spaziatura
	{#1 #2\thmnote{~\mdseries(#3)}} % formattazione nota
\theoremstyle{mio}
\newtheorem{teor}{Theorem}[section]
\newtheorem{cor}[teor]{Corollary}
\newtheorem{prop}[teor]{Proposition}
\newtheorem{lemma}[teor]{Lemma}

\newtheoremstyle{definition2}%
	{}{} % spazio sopra e sotto%
	{}{} % corpo del testo, indentation
	{\bfseries}{.}{ } % titolo del teorema: tipo di testo, divisore, spaziatura
	{#1 #2\thmnote{\mdseries~ #3}} % formattazione nota
\theoremstyle{definition2}

\newtheorem{oss}[teor]{Remark}

\newcounter{contacasi}%[section]
\newcommand{\caso}[1]{\addtocounter{contacasi}{1}\smallskip\underline{\emph{Case \arabic{contacasi}:} #1}.~}

\newcommand{\apery}{\mathrm{Ap}}
\newcommand{\area}{\mathcal{A}}

\newcommand{\embdim}{\nu}
\newcommand{\mult}{\mu}

\pgfplotsset{compat=1.14}
\usepgfplotslibrary{fillbetween}
\usetikzlibrary{intersections}
\pgfdeclarelayer{bg}
\pgfsetlayers{bg,main}
\usetikzlibrary{patterns}

\title[Wilf's conjecture with large second generator]{Wilf's conjecture for numerical semigroups with large second generator}
\author{Dario Spirito}
\email{spirito@mat.uniroma3.it}
\address{Dipartimento di Matematica e Fisica, Universit\`a degli Studi ``Roma Tre'', Roma, Italy}

\begin{document}
\begin{abstract}
We study Wilf's conjecture for numerical semigroups $S$ such that the second least generator $a_2$ of $S$ satisfies $a_2>\frac{c(S)+\mult(S)}{3}$, where $c(S)$ is the conductor and $\mult(S)$ the multiplicity of $S$. In particular, we show that for these semigroups Wilf's conjecture holds when the multiplicity is bounded by a quadratic function of the embedding dimension.
\end{abstract}

\maketitle

\section{Introduction and preliminaries}
A \emph{numerical semigroup} is a subset $S\subseteq\insN$ that contains 0, is closed under addition and such that the complement $\insN\setminus S$ is finite. In particular, there is a largest integer not contained in $S$, which is called the \emph{Frobenius number} of $S$ and is denoted by $F(S)$. The \emph{conductor} of $S$ is defined as $c(S):=F(S)+1$, and it is the minimal integer $x$ such that $x+\insN\subseteq S$. Calculating $F(S)$ is a classical problem (called the \emph{Diophantine Frobenius problem}), introduced by Sylvester \cite{sylvester_1884}; see \cite{ramirez-diophantine} for a general overview.

Given coprime integers $a_1<\ldots<a_n$, the numerical semigroup \emph{generated} by $a_1,\ldots,a_n$ is the set
\begin{equation*}
\langle a_1,\ldots,a_n\rangle:=\{\lambda_1a_1+\cdots+\lambda_na_n\mid \lambda_i\inN\}.
\end{equation*}
Conversely, if $S$ is a numerical semigroup, there are always a finite number of integers $a_1,\ldots,a_n$ such that $S=\langle a_1,\ldots,a_n\rangle$; moreover, there is a unique minimal set of such integers, whose cardinality, called the \emph{embedding dimension} of $S$, is denoted by $\embdim(S)$. The integer $a_1$, the smallest minimal generator of $S$, is called the \emph{multiplicity} of $S$, and is denoted by $\mult(S)$.

In 1978, Wilf \cite{wilf} suggested a relationship between the conductor and the embedding dimension of $S$. More precisely, set
\begin{equation*}
L(S):=\{x\in S\mid 0\leq x<c(S)\}.
\end{equation*}
Wilf hypothesized that the inequality
\begin{equation*}
\embdim(S)|L(S)|\geq c(S)
\end{equation*}
holds for every numerical semigroup $S$; this question is known as \emph{Wilf's conjecture}. The conjecture is still unresolved in the general case, although there have been several partial results: for example, it has been proven that Wilf's conjecture holds when $\embdim(S)\leq 3$ \cite{sylvester_1884,fgh_semigruppi}, when $|\insN\setminus S|\leq 60$ \cite{fromentin-hivert}, when $c(S)\leq 3\mult(S)$ \cite{kaplan,eliahou-wilf} and when $\embdim(S)\geq\mult(S)/2$ \cite{sammartano-wilf1}.

In this paper, we study Wilf's conjecture when $a_2$, the second smallest generator of $S$, is large, in the sense that
\begin{equation*}
a_2>\frac{c(S)+\mult(S)}{3}.
\end{equation*}
 In Section \ref{sect:bounds} we prove some bounds that this condition imposes on $S$, while in Section \ref{sect:L(S)} we estimate the cardinality of $L(S)$; finally, in Section \ref{sect:main} we show that for these semigroups Wilf's conjecture holds when $\embdim$ is large and the multiplicity is smaller than a quadratic function of the embedding dimension (Theorem \ref{teor:bound}). The basic idea is to split the generators of $S$ according to whether they are smaller or bigger than $\frac{c(S)+\mult(S)}{2}$, and using this division to estimate the cardinality of $L(S)$.

For general information and results about numerical semigroups, the reader may consult \cite{rosales_libro}.

\section{Splitting the generators}\label{sect:bounds}
From now on, $S$ will be a numerical semigroup, $\mult:=\mult(S)$ its multiplicity, $\embdim:=\embdim(S)$ its embedding dimension, and $c:=c(S)$ its conductor. We denote by $\apery(S)$ the \emph{Ap\'ery set} of $S$ with respect to its multiplicity, i.e.,
\begin{equation*}
\apery(S):=\{i\in S\mid i-\mult\notin S\}.
\end{equation*}
We recall that, for every $t\in\{0,\ldots,\mult-1\}$, there is a unique $x\in\apery(S)$ such that $x\equiv t\bmod\mult$; in particular, $\apery(S)$ has cardinality $\mult$. Note also that, since $c(S)$ is the maximal integer not belonging to $S$, every element of $\apery(S)$ is smaller than $c+\mult$.

Let now $P:=\{a_1,\ldots,a_\embdim\}$ be the set of minimal generators of $S$, with $\mult=a_1<a_2<\cdots<a_\embdim$. We shall always suppose that $a_2>\frac{c+\mult}{3}$. Since each $x\in P\setminus\{\mult\}$ belongs to $\apery(S)$, we can subdivide $P\setminus\{\mult\}$ into the following three sets:
\begin{align*}
P_1 & :=\left\{a\in P\setminus\{\mult\}\mid \inv{3}(c+\mult)<a<\inv{2}(c+\mult)\right\},\\
P_2 & :=\left\{a\in P\setminus\{\mult\}\mid \inv{2}(c+\mult)\leq a<\frac{2}{3}(c+\mult)\right\},\\
P_3 & :=\left\{a\in P\setminus\{\mult\}\mid \frac{2}{3}(c+\mult)\leq a<c+\mult\right\}.
\end{align*}
We set $q_i:=|P_i|$, for $i\in\{1,2,3\}$.

Let $\pi:\insZ\longrightarrow\insZ/\mult\insZ$ be the canonical quotient map, and let $A:=\pi(P)$, $A_i:=\pi(P_i)$. Given two subsets $X,Y\subseteq\insZ/\mult\insZ$, the \emph{sumset} of $X$ and $Y$ is
\begin{equation*}
X+Y:=\{x+y\mid x\in X, y\in Y\}.
\end{equation*}

\begin{prop}\label{prop:ricoprimento}
Let $S=\langle a_1,a_2,\ldots,a_\embdim\rangle$ be a numerical semigroup with $a_2>\frac{c(S)+\mult(S)}{3}$. Then, $\insZ/\mult\insZ=A\cup(A_1+A_1)\cup(A_1+A_2)$.
\end{prop}
\begin{proof}
Let $x\in\apery(S)$, $x\neq 0$. Then, $x<c+\mult$ and $x$ is a sum of elements of $P\setminus\{\mult\}$ (since $x-n\mult\notin S$ for $n>0$). The sum of three elements of $P\setminus\{\mult\}$ is bigger than $c+\mult$, and thus cannot be equal to $x$; likewise, $x$ cannot be the sum of two elements of $P_2\cup P_3$, and it also cannot be the sum of an element of $P_1$ and an element of $P_3$. Hence, the unique possibilities are $x\in P$, $x\in P_1+P_1$, or $x\in P_1+P_2$. The claim follows by projecting onto $\insZ/\mult\insZ$.
\end{proof}

Using the previous proposition, we can relate quantitatively $\mult$, $\embdim$, $q_1$ and $q_2$.
\begin{prop}\label{prop:bounds-muq1}
Let $S=\langle a_1,a_2,\ldots,a_\embdim\rangle$ be a numerical semigroup with $a_2>\frac{c(S)+\mult(S)}{3}$. Then:
\begin{enumerate}[(a)]
\item\label{prop:bounds-muq1:mu} $\displaystyle{\mult\leq\embdim+\frac{q_1(q_1+1)}{2}+q_1q_2}$;
\item\label{prop:bounds-muq1:munu} $\displaystyle{\mult\leq\inv{2}\embdim(\embdim+1)}$;
\item\label{prop:bounds-muq1:q1theta} $\displaystyle{q_1\geq\frac{2\embdim-1-\sqrt{(2\embdim+1)^2-8\mult}}{2}}$.%\geq\frac{\mult-\embdim}{\embdim-1}}$.
\end{enumerate}
\end{prop}
\begin{proof}
By Proposition \ref{prop:ricoprimento}, we have
\begin{equation*}
\mult\leq|A|+|A_1+A_1|+|A_1+A_2|.
\end{equation*}
\ref{prop:bounds-muq1:mu} now follows from the inequalities $|A|=\embdim$, $|A_1+A_1|\leq q_1(q_1+1)/2$ (by symmetry) and $|A_1+A_2|\leq q_1q_2$.

Using the previous point and the fact that $q_1+q_2\leq\embdim-1$, we have
\begin{align*}
\mult & \leq\embdim+\frac{q_1(q_1+1)}{2}+q_1q_2\leq\\
& \leq\embdim+\frac{q_1(q_1+1)}{2}+q_1(\embdim-1-q_1)= \embdim-\inv{2}q_1^2+\left(\embdim-\inv{2}\right)q_1,
\end{align*}
and thus
\begin{equation}\label{eq:ineq-q1}
q_1^2-(2\embdim-1)q_1+2(\mult-\embdim)\leq 0.
\end{equation}
Therefore, the discriminant of the equation is nonnegative, that is,
\begin{equation*}
0\leq(2\embdim-1)^2-8(\mult-\embdim)=(2\embdim+1)^2-8\mult,
\end{equation*}
or equivalently
\begin{equation*}
\mult\leq\inv{2}\embdim^2+\inv{2}\embdim+\inv{8}.
\end{equation*}
Moreover, since $\mult$ and $\embdim$ are integers, so is $\inv{2}\embdim+\inv{2}\embdim=\frac{\embdim(\embdim+1)}{2}$, and thus we can discard the $\inv{8}$. This proves \ref{prop:bounds-muq1:munu}.

Under this condition, \eqref{eq:ineq-q1} holds for $q_-\leq q_1\leq q_+$, where
\begin{equation*}
q_-:=\frac{2\embdim-1-\sqrt{(2\embdim+1)^2-8\mult}}{2}\quad\text{and}\quad q_+:=\frac{2\embdim-1+\sqrt{(2\embdim+1)^2-8\mult}}{2}
\end{equation*}
are the solutions of the corresponding equation; hence, \ref{prop:bounds-muq1:q1theta} follows.%the first inequality of \ref{prop:bounds-muq1:q1theta} follows. (Note that $q_+\geq\embdim\geq q_1$ always, so we do not get a ``new'' upper bound.) The second inequality holds if and only if $\mult\leq\frac{\embdim(\embdim+1)}{2}$, which holds by the previous point.
\end{proof}

\begin{oss}
The bound $q_-$ may actually be negative: however, if $q_1=0$ then part \ref{prop:bounds-muq1:mu} shows that $\mult\leq\embdim$, and thus $\mult=\embdim$. In this case, $S$ is of maximal embedding dimension and Wilf's conjecture holds by \cite[Theorem 20 and Corollary 2]{fgh_semigruppi}.
\end{oss}

%\begin{comment}
Part \ref{prop:bounds-muq1:mu} of Proposition \ref{prop:bounds-muq1} can be represented in a graphical way. Fix two integers, $\mult$ and $\embdim$. The inequalities
\begin{itemize}
\item $q_1,q_2\geq 0$;
\item $q_1+q_2\leq\embdim-1$;
\item $\displaystyle{q_1\left(\inv{2}q_1+\inv{2}+q_2\right)\geq\mult-\embdim}$.
\end{itemize}
define a subset of the plane $q_1q_2$, which is bounded by two lines and an hyperbola; we denote it by $\area(\mult,\embdim)$, or simply $\area$ if there is no danger of confusion. The set is pictured in Figure \ref{fig:area}.

\begin{figure}
\begin{tikzpicture}
\def\embdimu{20};
\def\multu{100};
\def\intersez{(sqrt(8*\multu-8*\embdimu)-1)/2};

\begin{axis}[axis on top=false,axis x line=middle,axis y line=middle,xmin=2,xmax=23,ymax=20,xlabel={$q_1$},ylabel={$q_2$},xtick={4,6,8,10,12,14,16,18,20},ytick={2,4,6,8,10,12,14}]%,xtick=\empty,ytick=\empty,
\addplot [name path=iperbole,domain=2:\intersez,mark=none] {(\multu-\embdimu)/x-x/2-1/2};% node[pos=0.1,anchor=west]{$q_1\left(\frac{q_1}{2}-\frac{1}{2}+q_1\right)=\mult-\embdim$};
\addplot [name path=retta,domain=-1:\intersez,mark=none] {\embdimu-1-x}; %node[pos=0.2,anchor=north]{$q_1+q_2=\embdim-1$};

\addplot [name path=iperbole2,domain=\intersez:23,mark=none] {(\multu-\embdimu)/x-x/2-1/2};
\addplot [name path=retta2,domain=\intersez:23,mark=none] {\embdimu-1-x};

\addplot [name path=asse,domain=\intersez:23,mark=none] {0};

\addplot fill between[ 
    of = iperbole and retta,
    split, % calculate segments
    every even segment/.style = {transparent},
    every odd segment/.style  = {pattern=vertical lines}
  ];

\addplot fill between[ 
    of = asse and retta2,
    split, % calculate segments
    every even segment/.style = {pattern=vertical lines},
    every odd segment/.style  = {transparent}
  ];
\end{axis}
\end{tikzpicture}
\caption{The region $\area(100,20)$.}\label{fig:area}
\end{figure}

Then, if $S$ is a semigroup with multiplicity $\mult$, embedding dimension $\embdim$ and $a_2>\frac{c(S)+\mult(S)}{3}$, then the lattice points $(q_1,q_2)$ in $\area(\mult,\embdim)$ correspond to the possible cardinalities of the sets $P_1$ and $P_2$.
%\end{comment}

\section{Estimates on $|L(S)|$}\label{sect:L(S)}
\begin{lemma}\label{lemma:conta}
Let $x,y,b,p$ be real numbers, with $p>0$ and $x<y$, and let $A:=b+p\insZ:=\{b+pn\mid n\inZ\}$. Then:
\begin{enumerate}[(a)]
\item $\displaystyle{|A\cap[x,y)|\geq\left\lfloor\frac{y-x}{p}\right\rfloor}$;
\item if $x\in A$ and $y\notin A$, then $\displaystyle{|A\cap[x,y)|=\left\lfloor\frac{y-x}{p}\right\rfloor+1}$.
\end{enumerate}
\end{lemma}
\begin{proof}
Let $k:=\left\lfloor\frac{y-x}{p}\right\rfloor$. Then,
\begin{equation*}
x+kp\leq x+\frac{y-x}{p}\cdot p\leq y;
\end{equation*}
hence, the $k$ sets $[x,x+p),[x+p,x+2p),\ldots,[x+(k-1)p,x+kp)$ are disjoint subintervals of $[x,y)$. In each $[x+ip,x+(i+1)p)$ there is exactly one element of $A$; hence, $|A\cap[x,y)|\geq k$.

Moreover, if $x\in A$ then $x+kp\in A$; since $y\notin A$, then $x+kp\neq y$, and thus the interval $[x+kp,y)$ is nonempty and contains exactly one element of $A$ (namely, $x+kp$). Hence, $|A\cap[x,y)|=k+1$.
\end{proof}

Our goal is to estimate the cardinality of $L:=L(S)$. To this end, we introduce the following notation: if $x$ is an integer, let
\begin{equation*}
L_x:=\{a\in L\mid a\equiv x\bmod\mult\}.
\end{equation*}
Clearly, $L_x$ and $L_y$ are disjoint if $x\not\equiv y\bmod\mult$.

\begin{prop}\label{prop:|L|}
Let $S=\langle a_1,a_2,\ldots,a_\embdim\rangle$ be a numerical semigroup with $a_2>\frac{c(S)+\mult(S)}{3}$. Then,
\begin{equation}\label{eq:|L|}
|L(S)|\geq\left\lfloor\frac{c}{\mult}\right\rfloor+\left(\left\lfloor\inv{2}\frac{c}{\mult}-\inv{2}\right\rfloor+1\right)q_1+ \left(\left\lfloor\inv{3}\frac{c}{\mult}-\frac{2}{3}\right\rfloor+1\right)q_2.
\end{equation}
\end{prop}
\begin{proof}
We have
\begin{equation*}
|L(S)|=\sum_{x\in\apery(S)}|L_x|\geq|L_0|+\sum_{x\in P_1}|L_x|+\sum_{x\in P_2}|L_x|.
\end{equation*}
Suppose $x\in\apery(S)$. Then, $L_x=(x+\mu\insZ)\cap[x,c)$, and by Lemma \ref{lemma:conta} we have $\displaystyle{|L_x|\geq\left\lfloor\frac{c-x}{\mult}\right\rfloor}$. Hence, $\displaystyle{|L_0|\geq\left\lfloor\frac{c}{\mult}\right\rfloor}$, while if $x\in P_1$ then
\begin{equation*}
|L_x|\geq\left\lfloor\frac{c-\inv{2}(c+\mult)}{\mult}\right\rfloor= \left\lfloor\inv{2}\frac{c}{\mult}-\inv{2}\right\rfloor
\end{equation*}
and if $x\in P_2$ then
\begin{equation*}
|L_x|\geq\left\lfloor\frac{c-\frac{2}{3}(c+\mult)}{\mult}\right\rfloor\geq\left\lfloor\inv{3}\frac{c}{\mult}-\frac{2}{3}\right\rfloor.
\end{equation*}
Hence,
\begin{equation*}
|L(S)|\geq\left\lfloor\frac{c}{\mult}\right\rfloor+\left\lfloor\inv{2}\frac{c}{\mult}-\inv{2}\right\rfloor q_1+\left\lfloor\inv{3}\frac{c}{\mult}-\frac{2}{3}\right\rfloor q_2.
\end{equation*}
Furthermore, applying again Lemma \ref{lemma:conta}, for every $x\in\{0\}\cup P_1\cup P_2$, except possibly one (namely, the $x$ such that $c\equiv x\bmod\mult$), there is a further element in $L_x\cap[x,c)$; hence, we can add $q_1+q_2$ to the quantity on the right hand side. The claim follows.
\end{proof}

We can get a slightly better version by considering also the relationship between the elements of $P_1$ and $P_2$; for this, we shall modify an idea introduced by S. Eliahou in \cite{eliahou-talk}. Say that a pair $(a,b)\in P_1\times P_2$ is an \emph{Ap\'ery pair} if $a+b\in\apery(S)$: then, $a+b<c+\mult$, and applying Lemma \ref{lemma:conta} we get
\begin{align}\label{eq:LaLb}
\begin{split}
|L_a|+|L_b| & =\left\lfloor\frac{c-a}{\mult}\right\rfloor+1+\left\lfloor\frac{c-b}{\mult}\right\rfloor+1\geq\\
& \geq \frac{2c-(a+b)}{\mult}>\frac{c-\mult}{\mult}=\frac{c}{\mult}-1.
\end{split}
\end{align}
Since $|L_x|+|L_y|$ is an integer, and the inequality is strict, we have $|L_x|+|L_y|\geq\left\lfloor\frac{c}{\mult}\right\rfloor$; in particular, this is better than the number $\left\lfloor\frac{c}{2\mult}\right\rfloor+\left\lfloor\frac{c}{3\mult}\right\rfloor\approx \frac{5}{6}\frac{c}{\mult}$ which we would get by considering the two estimates separately.

Let $\Sigma$ be the set of Ap\'ery pairs. We say that a subset $\{(a_i,b_i)\}_{i=1}^n\subseteq\Sigma$ is \emph{independent} if $a_i\neq a_j$ and $b_i\neq b_j$ for every $i\neq j$. Denoting by $\sigma$ the maximal cardinality of an independent set of Ap\'ery pairs, we obtain a slightly better version of Proposition \ref{prop:|L|}.
\begin{prop}\label{prop:|L|coppie}
Let $S=\langle a_1,a_2,\ldots,a_\embdim\rangle$ be a numerical semigroup with $a_2>\frac{c(S)+\mult(S)}{3}$. Then,
\begin{equation}\label{eq:|L|coppie}
|L|\geq\left\lfloor\frac{c}{\mult}\right\rfloor(1+\sigma)+\left(\left\lfloor\inv{2}\frac{c}{\mult}-\inv{2}\right\rfloor+1\right)(q_1-\sigma)+ \left(\left\lfloor\inv{3}\frac{c}{\mult}-\frac{2}{3}\right\rfloor+1\right)(q_2-\sigma).
\end{equation}
\end{prop}
\begin{proof}
Take $\sigma$ independent Ap\`ery pairs $\{(a_t,b_t)\}_{i=1}^\sigma$, and write $P_1=\{a_1,\ldots,a_{\sigma},c_1,\ldots,c_r\}$, $P_2=\{b_1,\ldots,b_\sigma,d_1,\ldots,d_s\}$. Thus, we have
\begin{equation*}
|L|\geq |L_0|+\sum_{t=1}^\sigma(|L_{a_i}|+|L_{b_i}|)+\sum_{j=1}^r|L_{c_j}|+ \sum_{k=1}^s|L_{d_k}|.
\end{equation*}
Using the estimates in the proof of Proposition \ref{prop:|L|} and the inequality \eqref{eq:LaLb} we get our claim.
\end{proof}

Propositions \ref{prop:|L|} and \ref{prop:|L|coppie} can be used to obtain a lower bound on the function $\frac{\embdim(S)|L(S)|}{c(S)}$: if this bound is at least 1, then Wilf's conjecture holds for the semigroup $S$. One problem lies in the floor functions appearing in \eqref{eq:|L|} and \eqref{eq:|L|coppie}; the simplest way to get rid of them is to use the inequality $\lfloor x\rfloor\geq x-1$. However, with some additional work we can obtain better estimates.

Indeed, observe that, if $c=(6k-1)\mult$ (where $k$ is an integer), then the quantities $\frac{c}{\mult}$, $\inv{2}\frac{c}{\mult}-\inv{2}$ and $\inv{3}\frac{c}{\mult}-\frac{2}{3}$ appearing in \eqref{eq:|L|coppie} are integers; this suggests to write $c$ as $(6k-1)\mult+\theta\mult$, where $k$ is an integer and $\theta\in[0,6)$ is a rational number. In this way, we have
\begin{equation*}
\left\lfloor\frac{c}{\mult}\right\rfloor=\left\lfloor\frac{(6k-1)\mult+\theta\mult}{\mult}\right\rfloor= 6k-1+\left\lfloor\theta\right\rfloor=\frac{c}{\mult}-(\theta-\left\lfloor\theta\right\rfloor);
\end{equation*}
analogously,
\begin{equation*}
\left\lfloor\inv{2}\frac{c}{\mult}-\inv{2}\right\rfloor+1= \frac{c}{2\mult}+\inv{2}-\left(\frac{\theta}{2}-\left\lfloor\frac{\theta}{2}\right\rfloor\right)
\end{equation*}
and
\begin{equation*}
\left\lfloor\inv{3}\frac{c}{\mult}-\frac{2}{3}\right\rfloor+1= \frac{c}{3\mult}+\inv{3}-\left(\frac{\theta}{3}-\left\lfloor\frac{\theta}{3}\right\rfloor\right).
\end{equation*}

Thus, when we multiply \eqref{eq:|L|coppie} by $\frac{\embdim}{c}$ we obtain
\begin{align*}
\frac{\embdim|L|}{c} & \geq\frac{\embdim}{\mult}\left[1-\frac{\mult}{c}(\theta-\lfloor\theta\rfloor)\right](1+\sigma)+ \frac{\embdim}{\mult}\left[\inv{2}-\frac{\mult}{c}\left(\frac{\theta}{2}-\left\lfloor\frac{\theta}{2}\right\rfloor-\inv{2}\right)\right](q_1-\sigma)+\\
& +\frac{\embdim}{\mult}\left[\inv{3}-\frac{\mult}{c}\left(\frac{\theta}{3}-\left\lfloor\frac{\theta}{3}\right\rfloor-\frac{1}{3}\right)\right](q_2-\sigma).
\end{align*}

We can write the right hand side of the previous inequality as
\begin{equation*}
\ell(q_1,q_2,\sigma):=\frac{\embdim}{\mult}[\alpha(1+\sigma)+\beta(q_1-\sigma)+\gamma(q_2-\sigma)],
\end{equation*}
where $\alpha,\beta,\gamma$ are rational numbers depending on $c$ and $\mult$. By \cite{eliahou-wilf}, Wilf's conjecture holds when $c\leq 3\mu$; hence, we can suppose, from now on, that $c>3\mu$. Let now
\begin{equation*}
l:=\begin{cases}
5 & \text{if~}\theta\in[0,4)\\
-1 & \text{if~}\theta\in[4,6).
\end{cases}
\end{equation*}
Then, $c\geq(l+\theta)\mult$; equivalently, $\frac{\mult}{c}\leq\inv{l+\theta}$. Therefore,
\begin{equation*}
\alpha\geq1-\frac{\theta-\lfloor\theta\rfloor}{l+\theta}.
\end{equation*}
If $k>-l$, the function $\displaystyle{x\mapsto\frac{x-k}{x+l}}$ is increasing for $x>-l$; hence, in the interval $[k,k+1)$ it is bounded above by its value at $x=k+1$. Thus,
\begin{equation*}
\alpha\geq 1-\frac{1}{l+1+\lfloor\theta\rfloor}.
\end{equation*}

A completely analogous reasoning can be used to estimate $\beta$ and $\gamma$, although in this case the calculations must consider the residue class of $\lfloor\theta\rfloor$ modulo 2 and 3 (for $\beta$ and $\gamma$, respectively). We obtain the following inequalities.
\begin{align*}
\beta &\geq\begin{cases}\inv{2} & \text{if~}\lfloor\theta\rfloor\equiv 0\bmod 2\\
\inv{2}-\inv{2(l+1+\lfloor\theta\rfloor)}& \text{if~}\lfloor\theta\rfloor\equiv 1\bmod 2
\end{cases}\\
\gamma&\geq\begin{cases}\inv{3} & \text{if~}\lfloor\theta\rfloor\equiv 0\bmod 3\\
\inv{3}-\inv{3(l+1+\lfloor\theta\rfloor)}& \text{if~}\lfloor\theta\rfloor\equiv 1\bmod 3\\
\inv{3}-\frac{2}{3(l+1+\lfloor\theta\rfloor)} &\text{if~}\lfloor\theta\rfloor\equiv 2\bmod 3
\end{cases}
\end{align*}

We now use this estimates to specialize \eqref{eq:|L|coppie} to each possible $\lfloor\theta\rfloor$.

\begin{description}[itemsep=\medskipamount]
\item[$\theta\in[0,1)$] $\displaystyle{\begin{aligned}[t]
\frac{\embdim|L|}{c} & \geq\frac{\embdim}{\mult}\left(\frac{5}{6}(1+\sigma)+\inv{2}(q_1-\sigma)+\inv{3}(q_2-\sigma)\right)=\\
& = \frac{\embdim}{\mult}\left(\frac{5}{6}+\inv{2}q_1+\inv{3}q_2\right)=:\frac{\embdim}{\mult}\ell_1(q_1,q_2,\sigma).
\end{aligned}}$

\item[$\theta\in[1,2)$] $\displaystyle{\begin{aligned}[t]
\frac{\embdim|L|}{c} & \geq\frac{\embdim}{\mult}\left(\frac{6}{7}(1+\sigma)+\frac{3}{7}(q_1-\sigma)+\frac{2}{7}(q_2-\sigma)\right)=\\
& = \frac{\embdim}{\mult}\left(\frac{6}{7}+\frac{3}{7}q_1+\frac{2}{7}q_2+\inv{7}\sigma\right)=:\frac{\embdim}{\mult}\ell_2(q_1,q_2,\sigma).
\end{aligned}}$

\item[$\theta\in[2,3)$] $\displaystyle{\begin{aligned}[t]
\frac{\embdim|L|}{c} & \geq\frac{\embdim}{\mult}\left(\frac{7}{8}(1+\sigma)+\inv{2}(q_1-\sigma)+\inv{4}(q_2-\sigma)\right)=\\
& = \frac{\embdim}{\mult}\left(\frac{7}{8}+\inv{2}q_1+\inv{4}q_2+\inv{8}\sigma\right)=:\frac{\embdim}{\mult}\ell_3(q_1,q_2,\sigma).
\end{aligned}}$

\item[$\theta\in[3,4)$] $\displaystyle{\begin{aligned}[t]
\frac{\embdim|L|}{c} & \geq\frac{\embdim}{\mult}\left(\frac{8}{9}(1+\sigma)+\frac{4}{9}(q_1-\sigma)+\inv{3}(q_2-\sigma)\right)=\\
& = \frac{\embdim}{\mult}\left(\frac{8}{9}+\frac{4}{9}q_1+\inv{3}q_2+\inv{9}\sigma\right)=:\frac{\embdim}{\mult}\ell_4(q_1,q_2,\sigma).
\end{aligned}}$

\item[$\theta\in[4,5)$] $\displaystyle{\begin{aligned}[t]
\frac{\embdim|L|}{c} & \geq\frac{\embdim}{\mult}\left(\frac{3}{4}(1+\sigma)+\inv{2}(q_1-\sigma)+\inv{4}(q_2-\sigma)\right)=\\
& = \frac{\embdim}{\mult}\left(\frac{3}{4}+\inv{2}q_1+\inv{4}q_2\right)=:\frac{\embdim}{\mult}\ell_5(q_1,q_2,\sigma).
\end{aligned}}$

\item[$\theta\in[5,6)$] $\displaystyle{\begin{aligned}[t]
\frac{\embdim|L|}{c} & \geq\frac{\embdim}{\mult}\left(\frac{4}{5}(1+\sigma)+\frac{2}{5}(q_1-\sigma)+\frac{4}{15}(q_2-\sigma)\right)=\\
& = \frac{\embdim}{\mult}\left(\frac{4}{5}+\frac{2}{5}q_1+\frac{4}{15}q_2+\frac{2}{15}\sigma\right)=:\frac{\embdim}{\mult}\ell_6(q_1,q_2,\sigma).
\end{aligned}}$
\end{description}

\section{Wilf's conjecture for large second generator}\label{sect:main}
Proposition \ref{prop:|L|coppie} isn't really better than Proposition \ref{prop:|L|} if we don't have a way to estimate $\sigma$. We do it in the following proposition, using a graph-theoretic method; see e.g. \cite{matching_theory} for the terminology used in the proof. The proof is inspired by \cite{eliahou-talk}.
\begin{prop}\label{prop:Ssigma}
Let $\Sigma$ and $\sigma$ as in Section \ref{sect:L(S)}. Then, $\displaystyle{\sigma\geq\frac{|\Sigma|}{\max\{q_1,q_2\}}}$.
\end{prop}
\begin{proof}
Define a graph $G$ by taking the disjoint union $P_1\sqcup P_2$ as the set of vertices and $\Sigma$ as the set of edges. Then, an independent subset of $\Sigma$ is exactly an independent subset of edges of $G$, that is, a matching, and $\sigma$ is exactly the matching number of $G$.

Moreover, $G$ is a bipartite graph, and thus (by K\"onig's theorem, see e.g. \cite[Theorem 1.1.1]{matching_theory}) the matching number of $G$ is equal to the its point covering number, i.e., to the cardinality of the smallest set $S\subseteq V(G)$ such that every edge of $G$ has a vertex in $S$.

For every $v\in V(G)$, the number of edges incident to $v$ is at most $q_1$ if $v\in P_2$ and at most $q_2$ if $v\in P_1$; hence, the point covering number of $G$ is at least $|E(G)|/\max\{q_1,q_2\}$. The claim follows.
\end{proof}

We also obtain a slightly better version of Proposition \ref{prop:bounds-muq1}\ref{prop:bounds-muq1:mu}.
\begin{cor}\label{cor:boundsigma}
Let $S=\langle a_1,a_2,\ldots,a_\embdim\rangle$ be a numerical semigroup with $a_2>\frac{c(S)+\mult(S)}{3}$, and let $\sigma$ as above. Then,
\begin{equation}\label{eq:boundsigma}
\frac{q_1(q_1+1)}{2}+\sigma\cdot\max\{q_1,q_2\}+\embdim\geq\mult.
\end{equation}
\end{cor}
\begin{proof}
Following the proof of Proposition \ref{prop:ricoprimento}, we see that if $x\in \apery(S)\cap(P_1+P_2)$ then $x=a_1+b_1$ for some Ap\'ery pair $(a_1,b_1)\in\Sigma$; hence, 
\begin{equation*}
|\apery(S)\cap(P_1+P_2)|\leq\Sigma\leq\sigma\cdot\max\{q_1,q_2\},
\end{equation*}
with the last inequality coming from Proposition \ref{prop:Ssigma}. The claim now follows using the proof of Proposition \ref{prop:bounds-muq1}\ref{prop:bounds-muq1:mu}.
\end{proof}

Before presenting the main theorem, we prove a lemma.
\begin{lemma}\label{lemma:stima}
Let $f(x,y):=\alpha+\beta x+\gamma y$, where $\alpha,\beta,\gamma$ are positive real numbers such that $\alpha\leq 1$ and $2\beta\geq\gamma$. For every $\epsilon>0$ there is a $\embdim_0(\epsilon)$ such that, if $\embdim\geq\embdim_0(\epsilon)$ and $\mult$ satisfies
\begin{equation}\label{eq:numult}
2\embdim\leq\mult<2\gamma(2\beta-\gamma)\embdim^2+(2\alpha-\gamma-1)\embdim-\frac{(2-2\alpha+\gamma)^2}{8\gamma(2\beta-\gamma)}-\epsilon,
\end{equation}
then
\begin{equation*}
f(x,y)\geq\frac{\mult}{\embdim}
\end{equation*}
for every $(x,y)\in\Omega:=\{(x,y)\in\insR^2\mid x>0,y\geq 0,~x\left(\inv{2}x+\inv{2}+y\right)+\embdim\geq\mult\}$.
\end{lemma}
\begin{proof}
Since $f$ is a linear function, and the components of its gradient are positive, the (eventual) minimum of $f$ on $\Omega$ can be reached only on its border $\mathcal{I}$, which is formed by a subset of an hyperbola (say $\mathcal{I}'$) and a subset of the $x$-axis. Moreover, $f$ is monotone increasing on the $x$-axis, and thus the minimum can only be reached on $\mathcal{I}'$. On it, $f$ becomes a quadratic function such that $f\to\infty$ when $x\to 0$ (since $y\to\infty$); therefore, $f$ has actually a minimum on $\mathcal{I}'$, and the point $(x_0,y_0)$ where it is reached satisfies, by Lagrange multipliers,
\begin{equation*}
\begin{cases}
\partial_xf(x_0,y_0)=x_0+\inv{2}+y_0=\beta\lambda\\
\partial_yf(x_0,y_0)=x_0=\gamma\lambda
\end{cases}
\end{equation*}
for some $\lambda\inR$; imposing $(x_0,y_0)\in\mathcal{I}'$ we have
\begin{align*}
\mult-\embdim &=x_0\left(\inv{2}x_0+\inv{2}+y_0\right)=\\
&=\partial_yf(x_0,y_0)\cdot\left(\partial_xf(x_0,y_0)-\inv{2}\partial_yf(x_0,y_0)\right)=\frac{\gamma(2\beta-\gamma)}{2}\lambda^2
\end{align*}
and thus
\begin{equation*}
\lambda=\sqrt{\frac{2}{\gamma(2\beta-\gamma)}}\sqrt{\mult-\embdim}.
\end{equation*}
Substituting in $f$, we have
\begin{align*}
f(x_0,y_0) &=\alpha+\beta\gamma\lambda+\gamma\left[\left(\beta-\gamma\right)\lambda-\inv{2}\right]=\\ 
&=\alpha-\frac{\gamma}{2}+\gamma\left(\beta+\beta-\gamma\right)\sqrt{\frac{2}{\gamma(2\beta-\gamma)}}\sqrt{\mult-\embdim}=\\
&=\alpha-\frac{\gamma}{2}+\sqrt{2\gamma(2\beta-\gamma)}\sqrt{\mult-\embdim}.
\end{align*}

Therefore, if $f(x_0,y_0)\geq\frac{\mult}{\embdim}$ then also $f(x,y)\geq\frac{\mult}{\embdim}$ for every $(x,y)\in\Omega$. We thus must solve an inequality in the form
\begin{equation}\label{eq:zetaxi}
\zeta+\xi\sqrt{\mult-\embdim}\geq\frac{\mult}{\embdim},
\end{equation}
or equivalently (since $\embdim>0$)
\begin{equation*}
\xi\embdim\sqrt{\mult-\embdim}\geq\mult-\zeta\embdim.
\end{equation*}
In our hypothesis, $\zeta=\alpha-\frac{\gamma}{2}\leq\alpha\leq 1$ and $\mult>\embdim$; hence, the right hand side is positive and we can square both sides, obtaining
\begin{equation*}
\xi^2\embdim^2(\mult-\embdim)\geq \mult^2-2\zeta\embdim\mult+\zeta^2 \embdim^2,
\end{equation*}
or, equivalently,
\begin{equation}\label{eq:diseq-xizeta}
\mult^2-(2\zeta\embdim+\xi^2\embdim^2)\mult+\zeta^2\embdim^2+\xi^2\embdim^3\leq 0.
\end{equation}

Suppose $\mult=2\embdim$: then, the left hand side of \eqref{eq:diseq-xizeta} is equal to
\begin{equation*}
\embdim^2(4-4\zeta+\zeta^2)+\embdim^3(-2\xi^2+xi^2)=\embdim^2[(1-\zeta)^2-\embdim\xi^2],
\end{equation*}
which is negative for $\embdim>\frac{(1-\zeta)^2}{\xi^2}$. Hence, under this condition the left hand side of \eqref{eq:diseq-xizeta} has two roots, $\mult_-<\mult_+$, and $\mult_-<2\embdim$. On the other hand,
\begin{align*}
\mult_+ &=\frac{(2\zeta\embdim+\xi^2\embdim^2)+\sqrt{\embdim^2\xi^2[4(\zeta-1)\embdim+\xi^2\embdim^2]}}{2}=\\
&=\frac{2\zeta\embdim+\xi^2\embdim^2+\xi^2\embdim^2\sqrt{1-\frac{4(1-\zeta)}{\xi^2\embdim}}}{2}.
\end{align*}
Expanding $\sqrt{1-\frac{4(1-\zeta)}{\xi^2\embdim}}$ as a Taylor series we have
\begin{align*}
\xi^2\embdim^2\sqrt{1-\frac{4(1-\zeta)}{\xi^2\embdim}}&= \xi^2\embdim^2\left(1-\inv{2}\cdot\frac{4(1-\zeta)}{\xi^2\embdim}-\inv{8}\left(\frac{4(1-\zeta)}{\xi^2\embdim}\right)^2+R_2(x)\right)=\\
&=\xi^2\embdim^2-2(1-\zeta)\embdim-\frac{2(1-\zeta)^2}{\xi^2}+\xi^2\embdim^2R_2(x),
\end{align*}
where $R_2$ is the remainder and $x=\frac{4(1-\zeta)}{\xi^2\embdim}$. In particular, $R_2(x)=O(x^3)$; hence, $\xi^2\embdim^2R_2(x)$ is $O(1/\embdim)$, and thus it is bigger than $-\epsilon$ for every $\embdim\geq\embdim_0(\epsilon)$ (for any $\epsilon>0$). Hence, for $\embdim\geq\embdim_0(\epsilon)$ \eqref{eq:zetaxi} holds for $\mult_-\leq\mult\leq\mult_+$, with
\begin{equation*}
\mult_+\geq 
\xi^2\embdim^2+(2\zeta-1)\embdim-\frac{(1-\zeta)^2}{\xi^2}-\epsilon.
\end{equation*}
Substituting $\zeta$ and $\xi$ with their definitions we have our claim.
\end{proof}

\begin{oss}\label{oss:resto}
The remainder of the Taylor series can actually be estimated fairly simply. Indeed, using $\frac{\mathrm{d}^3}{\mathrm{d}x^3}\sqrt{1-x}=-\frac{3}{8}\inv{(1-x)^{5/2}}$ and Taylor's theorem, we obtain, putting $\lambda:=\frac{4(1-\zeta)}{\xi^2}$,
\begin{equation*}
|\xi^2\embdim^2R_2(x)|\leq\frac{\lambda^3\xi^2}{16}\cdot\inv{\embdim}\left(\frac{\embdim}{\embdim-\lambda}\right)^{5/2}.
\end{equation*}
As a function of $\embdim$, the quantity on the right hand side is decreasing for $\embdim>\lambda$; for example, for $\embdim\geq 2\lambda$ we have
\begin{equation*}
|\xi^2\embdim^2R_2(x)|\leq\frac{\sqrt{2}\lambda^3\xi^2}{8\nu}.
\end{equation*}
We will use this estimate in Proposition \ref{prop:bound-explicit}.
\end{oss}

We are now ready to prove the main theorem.
\begin{teor}\label{teor:bound}
For every $\epsilon>0$ there is a $\embdim_0(\epsilon)$ such that, if $S=\langle a_1,a_2,\ldots,a_\embdim\rangle$ is a numerical semigroup such that:
\begin{itemize}
\item $a_2>\frac{c(S)+\mult(S)}{3}$,
\item $\embdim(S)=\embdim\geq\embdim_0(\epsilon)$, and
\item $\mult(S)\leq\frac{8}{25}\embdim^2+\frac{1}{5}\embdim-\frac{1}{2}-\epsilon$,
\end{itemize}
then $S$ satisfies Wilf's conjecture.
\end{teor}
\begin{proof}
By \cite{eliahou-wilf}, we need only to consider semigroups $S$ such that  $c>3\mult$. Write $c=(6k-1)\mult+\theta\mult$, where $k$ an integer and $\theta\in[0,6)$. By the discussion in Section \ref{sect:L(S)}, for every $i:=\lfloor\theta\rfloor$ there is a linear function $\ell_i(q_1,q_2,\sigma)$, not depending on $S$, such that
\begin{equation*}
\frac{\embdim|L|}{c}\geq\frac{\mult}{\embdim}\ell_i(q_1,q_2,\sigma).
\end{equation*}
We distinguish two cases.

\smallskip

Suppose $q_1\geq q_2$. By Corollary \ref{cor:boundsigma}, we have $\frac{q_1(q_1+1)}{2}+\sigma q_1+\embdim\geq\mult$; equivalently, the point $(q_1,\sigma)$ belongs to the set $\area(\embdim,\mult)$ defined at the end of Section \ref{sect:bounds}. Since $q_2\geq\sigma$, we have $\ell_i(q_1,q_2,\sigma)\geq\ell_i(q_1,\sigma,\sigma)=:\ell'_i(q_1,\sigma)$. By Lemma \ref{lemma:stima} applied to $\ell'_i$ (and since $\area(\embdim,\mult)\subseteq\Omega$), for every $\epsilon>0$ there is a $\overline{\embdim_i}(\epsilon)$ such that $\ell'_i(q_1,\sigma)\geq\frac{\embdim}{\mult}$ when $\embdim\geq\overline{\embdim_i}(\epsilon)$ and
\begin{equation*}
\mult\leq A_i\embdim^2+B_i\embdim+C_i-\epsilon,
\end{equation*}
where $A_i$, $B_i$ and $C_i$ are constants depending on $i$. In particular, $A_5$ is equal to $\frac{8}{25}$ and smaller than every other $A_i$; hence, there is a $\embdim'_0(\epsilon)$ such that $A_5\embdim^2+B_5\embdim+C_5-\epsilon\leq A_i\embdim^2+B_i\embdim+C_i-\epsilon$ for all $i$ and every $\embdim\geq \embdim'_0(\epsilon)$.

Therefore, $\frac{\embdim(S)|L(S)|}{c(S)}\geq 1$ for all semigroups $S$ with
\begin{itemize}
\item $a_2>\frac{c(S)+\mult(S)}{3}$,
\item $\mult(S)\leq A_5\embdim(S)^2+B_5\embdim(S)+C_5-\epsilon$ and
\item $\embdim(S)\geq\embdim_0(\epsilon):=\max\{\embdim'(\epsilon),\overline{\embdim_0}(\epsilon),\ldots,\overline{\embdim_5}(\epsilon)\}$.
\end{itemize}
Since the condition $\frac{\embdim(S)|L(S)|}{c(S)}\geq 1$ is equivalent to $S$ satisfying Wilf's conjecture, the claim follows substituting $A_5$, $B_5$ and $C_5$ with their value.

\medskip

Suppose $q_1\leq q_2$; by Corollary \ref{cor:boundsigma}, $\frac{q_1(q_1+1)}{2}+\sigma q_2+\embdim\geq\mult$. Then, $(q_1,q_2,\sigma)$ belongs to the set 
\begin{equation*}
\Omega':=\left\{(x,y,z)\in\insR^3\mid x,y,z\geq 0,z\leq x\leq y,~\frac{x(x+1)}{2}+yz\geq \mult-\embdim\right\}\subseteq\insR^3.
\end{equation*}
As in the proof of Lemma \ref{lemma:stima}, the minimum of $\ell_i$ on $\Omega'$ can only belong to the hyperboloid $\left\{\frac{x(x+1)}{2}+yz=\mult-\embdim\right\}$; by Lagrange multipliers, the minimum $(x_0,y_0,z_0)$ of $\ell_i$ on the hyperboloid satisfies
\begin{equation*}
\begin{cases}
\partial_x\ell_i(x_0,y_0,z_0)=x_0+\inv{2}=\beta_i\lambda\\
\partial_y\ell_i(x_0,y_0,z_0)=z_0=\gamma_i\lambda\\
\partial_z\ell_i(x_0,y_0,z_0)=y_0=\delta_i\lambda.
\end{cases}
\end{equation*}
Since $\gamma_i>\delta_i$ for each $i$, we must have $z_0>y_0$, which however implies that $(x_0,y_0,z_0)\notin\Omega'$; hence, the minimal point of $\ell_i$ in $\Omega'$ must belong on the intersection between the hyperboloid and one of the planes $\{x=z\}$ and $\{x=y\}$. If it is on the latter, we have $q_1=q_2$, and we fall back to the case $q_1\geq q_2$; if it is on the former, then we have to find the minimum of $\ell_i'(\sigma,q_2):=\ell_i(\sigma,q_2,\sigma)$ on 
\begin{equation*}
\Omega'':=\left\{(z,y)\in\insR^2\mid z>0,y\geq 0,z\leq y,~\frac{z(z+1)}{2}+yz\geq \mult-\embdim\right\}.
\end{equation*}
This set is contained in the domain $\Omega$ of Lemma \ref{lemma:stima}; hence, we can apply the lemma and, as in the proof of the case $q_1\geq q_2$, we obtain that $\frac{\embdim(S)|L(S)|}{c(S)}\geq 1$ for all semigroups $S$ with $\embdim(S)\geq\embdim_0(\epsilon)$ and $\mult(S)\leq A_j\embdim^2+B_j\embdim+C_j-\epsilon$ (where $\embdim_0(\epsilon),A_j,B_j,C_j$ are different from the previous case). However, a direct calculation shows that all $A_i$ are strictly bigger than $\frac{8}{25}$; hence, this case does not give any further restriction on the semigroups on which Wilf's conjecture holds (except perhaps the need to pass from $\embdim_0(\epsilon)$ to a larger number). Hence, the claim holds.
\end{proof}

The same reasoning can yield a more explicit version.
\begin{prop}\label{prop:bound-explicit}
Let $S=\langle a_1,a_2,\ldots,a_\embdim\rangle$ be a numerical semigroup with $\embdim(S)=\embdim\geq 10$. If $a_2>\frac{c(S)+\mult(S)}{3}$ and
\begin{equation*}
\mult(S)\leq\frac{8}{25}\embdim(S)^2+\frac{1}{5}\embdim(S)-\frac{5}{4},
\end{equation*}
then $S$ satisfies Wilf's conjecture.
\end{prop}
\begin{proof}
The proof is akin to the one of Theorem \ref{teor:bound}; we employ the same notation. Suppose $q_1\geq q_2$, and let $A_i,B_i,C_i$ be the coefficients of the polynomial in $\embdim$ which is on the right hand side of \eqref{eq:zetaxi} when $f=\ell'_i$. When $\embdim\geq 10$, for each $i$ the left hand side of \eqref{eq:diseq-xizeta} is negative when $\mult=2\embdim$; furthermore, $A_5\embdim^2+B_5\embdim+C_5\leq A_i\embdim^2+B_i\embdim+C_i$ for each $i$ when $\embdim\geq 10$. Moreover, in the notation of Remark \ref{oss:resto}, the largest $\lambda$ and $\lambda^3\xi^2$ appear again when $\theta\in[5,6)$, when their value is, respectively, 5 and 40; hence, the error term is at most
\begin{equation*}
\frac{\sqrt{2}\lambda^3\xi^2}{8\nu}\leq\frac{\sqrt{2}\cdot 40}{80}=\frac{\sqrt{2}}{2}<\frac{3}{4}.
\end{equation*}
Therefore, in this case Wilf's conjecture holds when 
\begin{equation*}
\mult(S)\leq\frac{8}{25}\embdim^2+\frac{1}{5}\embdim-\inv{2}-\frac{3}{4},
\end{equation*}
as claimed.

In the case $q_1\leq q_2$ the functions $\ell'_i$ we obtain putting $q_1=\sigma$ are always bigger than the corresponding functions for the case $q_1\geq q_2$; hence, also in this case Wilf's conjecture holds when $\embdim\geq 10$ and $\mult$ verifies the above inequality. The claim is proved.
\end{proof}

To conclude the paper, we give three variants of Theorem \ref{teor:bound} that can be proved with arguments very similar to the proof of the theorem. The first one looks at case $c\equiv 0\bmod\mult$, the second one strengthens the coefficients $\frac{8}{25}$ and the third one weakens Wilf's conjecture.
\begin{prop}
If $S=\langle a_1,a_2,\ldots,a_\embdim\rangle$ is a numerical semigroup such that
\begin{itemize}
\item $a_2>\frac{c(S)+\mult(S)}{3}$,
\item $\embdim(S)\geq 10$ and
\item $c\equiv 0\bmod\mult$,
\end{itemize}
then $S$ satisfies Wilf's conjecture.
\end{prop}
\begin{proof}
Using the same reasoning of the proof of Theorem \ref{teor:bound}, the worst bound of $\mult$ with respect to $\embdim$ happens in the case $q_1\geq q_2=\sigma$; under this condition, we have
\begin{equation*}
\frac{\embdim|L|}{c}\geq \frac{\embdim}{\mult}[\alpha(1+q_2)+\beta(q_1-q_2)]
\end{equation*}
where
\begin{equation*}
\alpha:=1-\frac{\mult}{c}(\theta-\lfloor\theta\rfloor)=1
\end{equation*}
(using the condition $c\equiv 0\bmod\mult$, which is equivalent to $\theta$ being an integer). Likewise,
\begin{equation*}
\beta:=\inv{2}-\frac{\mult}{c}\left(\frac{\theta}{2}-\left\lfloor\frac{\theta}{2}\right\rfloor-\inv{2}\right)\geq\inv{2}
\end{equation*}
because $\frac{\theta}{2}-\left\lfloor\frac{\theta}{2}\right\rfloor\leq\inv{2}$. Hence,
\begin{equation*}
\frac{\embdim|L|}{c}\geq \frac{\embdim}{\mult}\left[(1+q_2)+\inv{2}(q_1-q_2)\right]=1+\inv{2}q_1+\inv{2}q_2=:\frac{\embdim}{\mult}\ell(q_1,q_2).
\end{equation*}
By Lemma \ref{lemma:stima}, Remark \ref{oss:resto} and the proof of Proposition \ref{prop:bound-explicit}, if $\embdim(S)\geq 10$ then $\ell(q_1,q_2)\geq \mult/\embdim$ when $(q_1,q_2)\in\area(\mult,\embdim)$ and $\mult$ satisfies
\begin{equation*}
2\embdim\leq\mult<\inv{2}\embdim^2+\inv{2}\embdim-\inv{4}-\frac{\sqrt{2}}{2}.
\end{equation*}
Since $\mult$ is an integer and $\inv{4}+\frac{\sqrt{2}}{2}<1$, this means that Wilf's conjecture holds when $\mult<\inv{2}\embdim^2+\inv{2}\embdim$.

By Proposition \ref{prop:bounds-muq1}\ref{prop:bounds-muq1:munu}, the only case left to consider is $\mult=\inv{2}\embdim^2+\inv{2}\embdim=\frac{\embdim(\embdim+1)}{2}$. Under this condition, we have, by Proposition \ref{prop:bounds-muq1}\ref{prop:bounds-muq1:q1theta},
\begin{equation*}
q_1\geq\frac{2\embdim-1-1}{2}=\embdim-1;
\end{equation*}
since also $q_1\leq\embdim-1$ we must have $q_1=\embdim-1$ and $q_2=0$. In this case, 
\begin{equation*}
\frac{\embdim|L|}{c}\geq \frac{\embdim}{\mult}\left[1+\inv{2}(\embdim-1)\right]= \frac{\embdim}{\mult}\cdot\frac{\embdim+1}{2}=\frac{\embdim(\embdim+1)}{2\mult}=1
\end{equation*}
and thus $S$ satisfies Wilf's conjecture.
\end{proof}

\begin{prop}\label{prop:49qo}
There is an integer $N$ such that, for every $\embdim\geq N$, there are only finitely many numerical semigroups $S=\langle a_1,a_2,\ldots,a_\embdim\rangle$ with
\begin{itemize}
\item $a_2>\frac{c(S)+\mult(S)}{3}$,
\item $\embdim=\embdim(S)$, and
\item $\mult(S)\leq\frac{4}{9}\embdim^2$,
\end{itemize}
and that do \emph{not} satisfy Wilf's conjecture.
\end{prop}
%Note that this is not necessarily better than Theorem \ref{teor:bound}: if $\embdim$ is large, the theorem guarantees that a semigroup with large embedding dimension $\embdim$ satisfies Wilf's conjecture if $\mult\leq\frac{8}{25}\embdim^2$, while this proposition only proves that, if $\mult\leq\frac{4}{9}\embdim^2$, the number of eventual counterexamples is finite (and thus does not rule out that some of them satisfy $\mult\leq\frac{8}{25}\embdim^2$).
\begin{proof}
Fix any $\chi\in(0,1/3)$, and consider the function
\begin{equation*}
f(q_1,q_2):=1-\chi+\inv{2}q_1+\inv{3}q_2.
\end{equation*}
By Lemma \ref{lemma:stima}, for every $\epsilon>0$ there is an $N_1(\chi,\epsilon)$ such that, for every point $(q_1,q_2)\in\area(\mult,\embdim)$, with $\embdim\geq N_1(\chi,\epsilon)$, we have $f(q_1,q_2)\geq\mult/\embdim$ whenever
\begin{equation*}
\mult\leq\frac{4}{9}\embdim^2+\left(\frac{2}{3}-2\chi\right)\embdim-\inv{16}-\epsilon.
\end{equation*}
Let $N_2(\chi,\epsilon):=\left(\epsilon+\inv{16}\right)\left(\frac{2}{3}-2\chi\right)^{-1}$: then, for $\embdim\geq N_2(\chi,\epsilon)$, we have
\begin{equation*}
\left(\frac{2}{3}-2\chi\right)\embdim-\inv{16}-\epsilon\geq 0.
\end{equation*}
Therefore, for every $\embdim\geq N:=N(\chi,\epsilon):=\max\{N_1(\chi,\epsilon),N_2(\chi,\epsilon)\}$ we have $f(q_1,q_2)\geq\mult/\embdim$ whenever $\mult\leq\frac{4}{9}\embdim^2$. Equivalently, we have
\begin{equation*}
1+\inv{2}q_1+\inv{3}q_2\geq\frac{\mult}{\embdim}+\chi.
\end{equation*}

Using the inequality $\lfloor x\rfloor>x-1$ on Proposition \ref{prop:|L|}, we have
\begin{equation*}
\frac{\embdim|L|}{c}\geq\frac{\embdim}{\mult}\left(1+\inv{2}q_1+\inv{3}q_2\right)-\frac{\embdim}{c}\left(1+\inv{2}q_1+\frac{2}{3}q_2\right)
\end{equation*}
which for $\embdim\geq N$ is bigger than
\begin{equation*}
\frac{\embdim}{\mult}\left(\frac{\mult}{\embdim}+\chi\right)-\frac{\embdim}{c}\left(\frac{\mult}{\embdim}+\chi+\inv{3}q_2\right)\geq 1+\frac{\embdim}{\mult}\chi-\inv{c}\left(\mult+\chi\embdim+\frac{\embdim(\embdim-1)}{3}\right),
\end{equation*}
using also the fact that $q_2\leq\embdim-1$. The quantity on the right hand side is bigger than 1 when
\begin{equation*}
\frac{\embdim}{\mult}\chi-\inv{c}\left(\mult+\chi\embdim+\frac{\embdim(\embdim-1)}{3}\right)\geq 0;
\end{equation*}
since $c$, $\embdim$, $\mult$ and $\chi$ are positive, this is equivalent to
\begin{equation}\label{eq:boundchi}
c\geq\frac{\mult}{\chi\embdim}\left(\mult+\chi\embdim+\frac{\embdim(\embdim-1)}{3}\right),
\end{equation}
and all semigroups satisfying this inequality satisfy Wilf's conjecture. 

In particular, for any value of $\embdim$, $\mult$ and $\chi$, there are only a finite number of semigroups that do not satisfy this condition. For any $\embdim$, there are also a finite number of multiplicities $\mult$ satisfying $\mult\leq\frac{4}{9}\embdim^2$; hence, for any fixed $\embdim\geq N$ there are only finitely many numerical semigroups that verify the hypothesis of the theorem and that do not satisfy Wilf's conjecture.
\end{proof}

We note that the right hand side of \eqref{eq:boundchi} is very large: for example, if $\embdim=10$, $\mult=50$ and $\chi=\inv{6}$, then it is equal to 26050. The strategy used in the proof of Theorem \ref{teor:bound} (i.e., writing $c=(6k-1)\mult+\theta\mult$ and using different estimates for different $\lfloor\theta\rfloor$) can be employed to obtain numerically better bounds (but still with the hypothesis $\mult\leq\frac{4}{9}\embdim^2$).

\begin{prop}\label{prop:weakwilf}
For every $\lambda<\frac{4}{5}$ there is a $\embdim_0(\lambda)$ such that, if $S=\langle a_1,a_2,\ldots,a_\embdim\rangle$ is a numerical semigroup such that $a_2>\frac{c(S)+\mult(S)}{3}$ and $\embdim\geq\embdim_0(\lambda)$, then
\begin{equation}\label{eq:weakwilf}
\embdim(S)|L(S)|\geq \lambda\cdot c(S).
\end{equation}
\end{prop}
\begin{proof}
Fix a $\lambda<\frac{4}{5}$. Let $c=(6k-1)\mult+\theta\mult$, with $k$ an integer and $\theta\in[0,6)$. For any fixed $\lfloor\theta\rfloor$, we have
\begin{equation*}
\frac{\embdim|L|}{c}\geq\frac{\mult}{\embdim}(\alpha+\beta q_1+\gamma q_2+\delta\sigma),
\end{equation*}
for some $\alpha,\beta,\gamma,\delta$ depending on $\lfloor\theta\rfloor$. Therefore, \eqref{eq:weakwilf} holds if
\begin{equation*}
\lambda^{-1}\alpha+\lambda^{-1}\beta q_1+\lambda^{-1}\gamma q_2+\lambda^{-1}\delta\sigma\geq\frac{\mult}{\embdim}
\end{equation*}
which, by Lemma \ref{lemma:stima}, holds for
\begin{equation*}
\mult\leq [2(\lambda^{-1}\gamma)(2\lambda^{-1}\beta-\lambda^{-1}\gamma)-\epsilon]\embdim^2=\left(\frac{2\gamma(2\beta-\gamma)}{\lambda^2}-\epsilon\right)\embdim^2.
\end{equation*}
for $\embdim\geq\embdim'_0(\epsilon)$. By Theorem \ref{teor:bound}, $2\gamma(2\beta-\gamma)$ is at least $\frac{8}{25}$; if $\lambda<\frac{4}{5}$, then
\begin{equation*}
\frac{2\gamma(2\beta-\gamma)}{\lambda^2}>\frac{8}{25}\cdot\frac{25}{16\cdot 2}=\inv{2}.
\end{equation*}

Therefore, we can choose an $\epsilon$ satisfying
\begin{equation*}
0<\epsilon<\frac{2\gamma(2\beta-\gamma)}{\lambda^2}-\inv{2},
\end{equation*}
and for such an $\epsilon$ there is a $\embdim''_0(\epsilon,\lambda)$ such that
\begin{equation*}
\left(\frac{2\gamma(2\beta-\gamma)}{\lambda^2}-\epsilon\right)\embdim^2>\inv{2}\embdim^2+\inv{2}\embdim
\end{equation*}
for all $\embdim\geq\embdim''_0(\epsilon,\lambda)$. Setting $\embdim_0(\lambda):=\max\{\embdim'_0(\epsilon),\embdim''_0(\epsilon,\lambda)\}$, we have that the inequality \eqref{eq:weakwilf} holds for $\embdim\geq\embdim_0(\lambda)$ and $\mult\leq\inv{2}\embdim^2+\inv{2}\embdim$. Since every semigroup with $a_2>\frac{c(S)+\mult(S)}{3}$ satisfies the latter condition (by Proposition \ref{prop:bounds-muq1}\ref{prop:bounds-muq1:munu}), the claim holds.
\end{proof}

\bibliographystyle{plain}
\bibliography{/bib/articoli,/bib/libri,/bib/miei,/bib/eventualia}

\begin{thebibliography}{10}

\bibitem{eliahou-talk}
Shalom Eliahou.
\newblock A graph-theoretic approach to {W}ilf's conjecture.
\newblock Meeting of the Catalan, Spanish, Swedish Math Societies, Ume\aa,
  Sweden (June 12--15, 2017). http://www.ugr.es/~semigrupos/Umea-2017/.

\bibitem{eliahou-wilf}
Shalom Eliahou.
\newblock Wilf's conjecture and {M}acaulay's theorem.
\newblock arxiv:1703.01761 [math.CO], 2017.

\bibitem{fgh_semigruppi}
Ralf Fr{\"o}berg, Christian Gottlieb, and Roland H{\"a}ggkvist.
\newblock On numerical semigroups.
\newblock {\em Semigroup Forum}, 35(1):63--83, 1987.

\bibitem{fromentin-hivert}
Jean Fromentin and Florent Hivert.
\newblock Exploring the tree of numerical semigroups.
\newblock {\em Math. Comp.}, 85(301):2553--2568, 2016.

\bibitem{kaplan}
Nathan Kaplan.
\newblock Counting numerical semigroups by genus and some cases of a question
  of {W}ilf.
\newblock {\em J. Pure Appl. Algebra}, 216(5):1016--1032, 2012.

\bibitem{matching_theory}
L\'aszl\'o Lov\'asz and Michael~D. Plummer.
\newblock {\em Matching theory}, volume 121 of {\em North-Holland Mathematics
  Studies}.
\newblock North-Holland Publishing Co., Amsterdam; North-Holland Publishing
  Co., Amsterdam, 1986.
\newblock Annals of Discrete Mathematics, 29.

\bibitem{ramirez-diophantine}
Jorge~Luis Ram{\'{\i}}rez~Alfons{\'{\i}}n.
\newblock {\em The {D}iophantine {F}robenius problem}, volume~30 of {\em Oxford
  Lecture Series in Mathematics and its Applications}.
\newblock Oxford University Press, Oxford, 2005.

\bibitem{rosales_libro}
Jos\'e~Carlos Rosales and Pedro~A. Garc{\'{\i}}a-S{\'a}nchez.
\newblock {\em Numerical semigroups}, volume~20 of {\em Developments in
  Mathematics}.
\newblock Springer, New York, 2009.

\bibitem{sammartano-wilf1}
Alessio Sammartano.
\newblock Numerical semigroups with large embedding dimension satisfy {W}ilf's
  conjecture.
\newblock {\em Semigroup Forum}, 85(3):439--447, 2012.

\bibitem{sylvester_1884}
James~Joseph Sylvester.
\newblock Mathematical questions with their solutions.
\newblock {\em Educational Times}, 41:21, 1884.

\bibitem{wilf}
Herbert~S. Wilf.
\newblock A circle-of-lights algorithm for the ``money-changing problem''.
\newblock {\em Amer. Math. Monthly}, 85(7):562--565, 1978.

\end{thebibliography}

\end{document}